\documentclass[reqno]{amsart}
\usepackage{mathrsfs,amsmath}
\usepackage{bm}
\usepackage{pdfsync}

\newcommand{\dom}{\operatorname{dom}}
\newcommand{\myRe}{\operatorname{Re}}
\newcommand{\myIm}{\operatorname{Im}}

\newcommand{\bC}{\mathbb{C}}
\newcommand{\bN}{\mathbb{N}}
\newcommand{\bR}{\mathbb{R}}

\newcommand{\bpsi}{\bm{\psi}}

\newcommand\ri{\mathrm{i}}
\newcommand\re{\mathrm{e}}

\newcommand\rd{\mathrm{d}}

\newcommand\by{\mathbf{y}}
\newcommand\bw{\mathbf{w}}
\newcommand\bz{\mathbf{z}}

\newtheorem{theorem}{Theorem}
\newtheorem{lemma}{Lemma}
\newtheorem{proposition}{Proposition}
\newtheorem{definition}{Definition}
\title[Inverse scattering for energy-dependent Schr\"{o}dinger equations]{Inverse scattering on the half-line\\ for energy-dependent Schr\"{o}dinger equations}
\author[Hryniv]{Rostyslav O. Hryniv}
\address[Hryniv]{Ukrainian Catholic University, Lviv, Ukraine \and University of Rzesz\'ow, Poland}
\email{rhryniv@ucu.edu.ua, rhryniv@ur.edu.pl}

\author[Manko]{Stepan S.~Manko}
\address[Manko]{Doppler Institute for Mathematical Physics and Applied
Mathematics, Czech Technical University in Prague,
B\v{r}ehov\'{a} 7, 11519 Prague, Czech Republic and  Department of Physics, Faculty of Nuclear Science and Physical Engineering, Czech Technical University in Prague,
Pohrani\v{c}n\'{\i} 1288/1,  40501 D\v{e}\v{c}\'{\i}n,
Czech Republic\footnote{On leave of absence from Pidstryhach Institute for Applied Problems of Mechanics and Mathematics, 3b Naukova st., 79060 Lviv, Ukraine}}
\email{stepan.manko@gmail.com}



\begin{document}
\begin{abstract}
In this paper, we study the inverse scattering problem for energy-dependent Schr\"{o}dinger equations on the half-line with energy-dependent boundary conditions at the origin.
Under certain positivity and very mild regularity assumptions, we transform this scattering problem to the one for non-canonical Dirac systems and show that, in turn, the latter can be placed within the known scattering theory for ZS-AKNS systems. This allows us to give a complete description of the corresponding scattering functions $S$ for the class of problems under consideration and justify an algorithm of reconstructing the problem from~$S$.
\end{abstract}

%%%%%%%%%%%%%%%%%%%%%%%%%%%%%%%%%%%%%%%%%%%%%%%%%%

\maketitle
\section{Introduction}
The main aim of the paper is to develop the direct and inverse scattering theory for a class of one-dimensional energy-dependent Schr\"{o}dinger equations
\begin{equation}\label{eq:EDSE}
    -y''+qy+2kpy=k^2 y
\end{equation}
on the half-line with integrable potentials~$p$, highly singular potentials~$q$, and some boundary conditions at the origin.
%the energy-dependent boundary condition
%\begin{equation}\label{eq:DBC}
%\cos\alpha\, y^{[1]}(0)+\rk \sin\alpha\, y(0)=0.
%\end{equation}
Schr\"odinger equations of this form arise in various models of quantum and classical mechanics; for instance, the Klein--Gordon equation \cite{Jon93,Naj83} modeling interactions between colliding relativistic spinless particles (called anti-particles) with zero mass is a special case of \eqref{eq:EDSE} when $q=-p^2$. Further models leading to~\eqref{eq:EDSE} are mentioned at the end of this section. In quantum mechanical models, $k^2$ is related to the energy of the system, so that the formal potential~$q+ 2kp$ of the Schr\"odinger expression on the left-hand side of~\eqref{eq:EDSE} becomes energy dependent, thus giving the name to the equation.

Throughout the paper, we assume $p$ is real-valued and belongs to the space $X:=L^1(\bR_+)\cap L^2(\bR_+)$ while $q$ is a real-valued distribution belonging to the space $H^{-1}_{\mathrm{loc}}(\bR_+)$ and having the form $q=u'+u^2$ for some $u\in X$.
In other words, $q$ is a so-called Miura potential; see~\cite{KapPerShuTop05} for thorough discussion of Schr\"odinger operators with such potentials and Subsection~\ref{ssec:spectral} for justification of this assumption. We note that such $q$ may  contain e.g.\ the Dirac $\delta$-functions and local singularities of the Coulomb $1/x$-type that are widely used in quantum mechanics to model various interactions between or within the elementary particles.

Denote by $y^{[1]}:=y'-uy$ the so-called \emph{quasi-derivative} of a function~$y$; as suggested in~\cite{SavShk:1999,SavShk:2003}, the differential expression $-y'' + qy$ should then be interpreted as $-(y^{[1]})' - u y^{[1]}$ for the corresponding set of functions. We shall consider the general boundary conditions of the form
\begin{equation}\label{eq:bc}
    \sin\alpha\, y^{[1]}(0,k) + k \cos\alpha\, y(0,k)=0;
\end{equation}
note that for $\alpha=0$ this becomes the Dirichlet boundary condition and for $\alpha=\pi/2$ an analogue of the Neumann boundary conditions.

As in the standard case with regular~$q$ and $p\equiv0$, for real $k$ equation~\eqref{eq:EDSE} has the so-called Jost solution $f(\cdot,k)$ uniquely determined by the asymptotics $f(x,k)=
e^{ikx}(1+o(1))$ as $x\to\infty$.
Moreover, for nonzero real~$k$ the solutions $f(\cdot,k)$ and~$\overline{f(\cdot,k)}$ are linearly independent, and there is a number~$S(k)\in\mathbb{C}$ such that the solution $f(\cdot,k) + S(k) \overline{f(\cdot,k)}$ satisfies the boundary condition~\eqref{eq:bc}.
The function~$S$ is called the \emph{scattering function} of the problem~\eqref{eq:EDSE}, \eqref{eq:bc} and is a direct analogue of the scattering function for the classical Schr\"odinger scattering problem.

Note that the Jost solutions exist also for $k$ in the open upper half-plane~$\mathbb{C}_+$, and the \emph{Jost function}
\begin{equation}\label{eq:Jost-funct}
    s(k):=\sin\alpha\, f^{[1]}(0,k) + k \cos\alpha\, f(0,k)
\end{equation}
is analytic in $\mathbb{C}_+$ and continuous in the closure~$\overline{\mathbb{C}_+}$. A short computation shows that the scattering function~$S$ can be expressed in terms of the Jost function~$s$ as
\begin{equation*}
%\label{eq:scat-funct}
    S(k) = \frac{s(k)}{\overline{s(k)}}
        = \frac{\sin\alpha\, f^{[1]}(0,k) + k \cos\alpha\, f(0,k)}%
            {\sin\alpha\, \overline{f^{[1]}(0,k)} + k \cos\alpha\, \overline{f(0,k)}}.
\end{equation*}
In general, the Jost function~$s$ may have zeros in~$\mathbb{C}_+$, but in contrast with the classical Schr\"odinger scattering problem they need not belong to the imaginary axis. If $z_0\in\mathbb{C}_+$ is a zero of~$s$, then $f(\cdot,z_0)$ is a solution of~\eqref{eq:EDSE} belonging to the Hilbert space~$L^2(\bR_+)$ and satisfying the boundary condition~\eqref{eq:bc}, and in that case $z_0$ is called an \emph{eigenvalue}, or \emph{bound state} of the problem~\eqref{eq:EDSE}, \eqref{eq:bc}. The eigenvalues always come in complex conjugate pairs and, moreover, may have multiplicity larger than one. The scattering function~$S$ and the eigenvalues together  with the corresponding norming constants form the scattering data of the problem~\eqref{eq:EDSE}, \eqref{eq:bc}. We note, however, that under the standing assumptions on~$q$, the Jost function~$s$ is zero-free in~$\mathbb{C}_+$ and thus there are no eigenvalues (see Subsection~\ref{ssec:spectral}); as a result, the scattering data is just the scattering function~$S$.

The direct scattering problem consists in finding the scattering data for given potentials $p$, $q$, and the constant~$\alpha$ in the boundary conditions. The inverse scattering problem is to construct $p$, $q$, and $\alpha$ given the scattering data. A complete solution to the scattering problem for a given class of problems~\eqref{eq:EDSE}, \eqref{eq:bc} includes several subtasks, such as to characterise the corresponding set of scattering data, establish the reconstruction algorithm, prove uniqueness and continuous dependence.
In the classical setting with $p\equiv0$, $q$ of Faddeev--Marchenko class, and $\alpha=0$, a complete solution to the scattering problem has been known since the 1950-ies and is well explained in e.g. \cite{ChaSab89,Lev87,Mar86,New89}. We review below some related work on the inverse scattering in the energy-dependent case but observe here that satisfactory results have only been derived only when there are no eigenvalues, or bound states of~\eqref{eq:EDSE}--\eqref{eq:bc}.

Here, our aim is to give a complete solution of the inverse scattering problem for energy-dependent Schr\"odinger equation~\eqref{eq:EDSE} subject to the boundary conditions~\eqref{eq:bc} and under the weakest possible regularity assumptions on the potentials $p$ and~$q$. Under these assumptions, the problem possesses no bound states; we give a complete description of the scattering functions for the problem under consideration, prove existence and uniqueness of solution to the inverse scattering problem, and suggest a reconstruction algorithm. The main results of the paper are summarized in Theorem~\ref{thm:main} below. Our approach consists in transforming the scattering problem for energy-dependent Schr\"odinger equation~\eqref{eq:EDSE} to the one for non-canonical Dirac systems, and then placing the latter within the well developed scattering theory for ZS-AKNS systems. The method is motivated by a similar approach to problems on finite intervals realized in~\cite{HryPro12} and \cite{Pro13}. 

In the context of the scattering theory, the energy-dependent Schr\"{o}dinger equations \eqref{eq:EDSE} were considered first by Cornill~\cite{Cor70} and Weiss and Scharf~\cite{WeiSch71}, where the authors derived the Marchenko equation for~\eqref{eq:EDSE} and conjectured the applicability of the classical Gelfand--Levitan--Marchenko inverse scattering method developed in the works of Gelfand and Levitan \cite{GelLev51}, Krein \cite{Kre53,Kre55},
and Marchenko \cite{Mar52,Mar55}; see also the reviews \cite{DeiTru79,Fad59} and the books \cite{ChaSab89,Lev87,Mar86,New89}.

Jaulent and Jean \cite{Jau72,JauJea72} were seemingly the first to develop systematically the inverse scattering theory for Schr\"{o}dinger equations of the form \eqref{eq:EDSE}.
In the radial case, when there are no bound states,
the authors established a procedure of recovering the potentials $p$ and $q$ from the scattering data, under the assumption that $p$ and $q$ are real-valued and integrable along with their derivatives.
In the later publications \cite{JauJea76,JauJea76i} they also extended they approach to solve the inverse problem associated with the equation~\eqref{eq:EDSE} on the whole line.
Namely, they derived the Marchenko equation coupled with
a nonlinear ordinary differential equation for $\int_x^\infty p(t)\,\mathrm{d}t$.
Also, Jaulent used the generalization of that method to treat the case of purely imaginary $p$ in his short paper \cite{Jau76} containing no complete proofs.

Schr\"{o}dinger equations~\eqref{eq:EDSE} with $p$~being imaginary as well as real were studied by Sattinger and Szmigielski in \cite{SatSzi95,SatSzi96}.
The authors proved the invertibility of the scattering transform by means of a simple vanishing lemma when there are no bound states.
If the potential $p$ is real with zero mean, their method is simpler than the one of Jaulent and Jean, as it permits recovering $p$ by solving an algebraic equation rather than a differential equation.
Although in generic situation the scattering transform may not be invertible for one or more bound states, in \cite{SatSzi95} the authors solved the inverse problem for a single bound state and obtained a so-called breather solution.
In \cite{SatSzi95,SatSzi96} the inverse scattering problem is formulated as a Riemann--Hilbert problem on a Riemann surface.
As an application the authors proved a global existence theorem for a family of so-called isospectral flows.
The one-soliton solution of this nonlinear equation is similar in structure to the well-known breather solution of the sine-Gordon equation.
Isospectral flows for energy-dependent operators arise in various applications; for instance, an energy-dependent isospectral operator generates the Camassa--Holm equation that describes a special shallow water wave and is related to the Dym hierarchy \cite{CamHol93}.

Adapting an algorithm suggested originally by Deift and Trubowitz
\cite{DeiTru79}, Kami\-mura~\cite{Kam1,Kam3,Kam4} generalized the results of Jaulent and Jean to a wider class of potentials.
He studied the inverse problem on the half-line \cite{Kam4} as well as on the whole line \cite{Kam3} and also applies in~\cite{Kam2} the inverse scattering theory for \eqref{eq:EDSE} with $q=0$ to reconstruct an oceanic flow from the data of an observable property in the ocean.

Aktosun et al.~\cite{AktKlavdM98i,AktKlavdM98ii} described the wave propagation in a nonconservative medium by equation \eqref{eq:EDSE} with a real $q$ and a purely imaginary $p$.
In their model $\mathrm{Im}\,p$ represents the energy absorption
or generation, while $q$ represents the restoring force density.
The authors solved one-dimensional inverse scattering problem for equation \eqref{eq:EDSE} with bound states for integrable $p$ and $q$ from the Faddeev--Marchenko class under some additional restrictions on $p$ and $q$.
Observe that the direct and inverse scattering problems for equation \eqref{eq:EDSE} with real $p$ and with purely imaginary~$p$ are essentially different. For instance, in the latter case the differential operators arising in the left-hand side of~\eqref{eq:EDSE} are not self-adjoint even for real~$k$  and the corresponding scattering matrices are not unitary.

Kaup \cite{Kau75} treated the direct and inverse scattering problem for Schr\"{o}dinger equation \eqref{eq:EDSE} with the coefficients $p$ and $q$ tending to non-zero limits at infinity.
Such an equation may be reformulated as equation \eqref{eq:EDSE} with potentials of the Faddeev--Marchenko class but with $k^2$ replaced by $k^2+m$ and is connected with a long-wave approximation of Boussinesq-type equations.
The case when $m$ equals $1$ was studied in \cite{Tsu81}, where the author proved that the problem may be reduced to matrix analogue of
\eqref{eq:EDSE} without the mass parameter $m$.
Similar inverse scattering problem was treated in \cite{vdMPiv2001}.

Inverse scattering for energy-dependent Schr\"{o}dinger equations has also appeared in the literature in the context of the waves propagation in nondispersive media where the wave speed depends on position.
Such waves (e.g., sound, electromagnetic, or elastic ones)
may be described by an equation with a potential $k^2q$ that is proportional to the energy \cite{AktvdM90,AktvdM91}.
Inverse scattering problem for the Schr\"{o}dinger
equations with polynomial energy-dependent potentials were discussed in \cite{Nab06,NabGus06}.

The paper is organized as follows. In the next section, we explain how the scattering problem for the energy-dependent Schr\"{o}dinger equations~\eqref{eq:EDSE} canbe transformed to that for related non-canonical ZS-AKNS systems. Using this correspondence, In Section~\ref{sec:direct}, we characterize the set of scattering functions for the energy-dependent Schr\"{o}dinger equations under consideration, and then in Section~\ref{sec:inverse} we completely solve the inverse scattering problem and provide a reconstruction algorithm.

\section{Preliminaries}\label{sec:prelim}

\subsection{Miura potentials and their Riccati representatives}

We shall consider the case where the potential $q$ admits a Riccati representation given by the
Miura map \cite{Miu68}.
Recall that the {\it Miura map} is the nonlinear mapping
\begin{align*}
     M :L^2_{\mathrm{loc}}(\mathbb{R}) &\to H^{-1}_\mathrm{loc}(\mathbb{R}),\\
            u &\mapsto u'+ u^2.
\end{align*}
The Miura map has played a fundamental role in the study of existence and well-posedness questions for KdV and mKdV equations, since it relates smooth solutions of these equations, see \cite{Miu68}.

The range of the Miura map may be characterized as follows.
For $q\in H^{-1}_\mathrm{loc}(\mathbb{R})$
real-valued and $\varphi\in C_0^\infty(\mathbb{R})$, consider the Schr\"{o}dinger form
\[
    \mathfrak{q}(\varphi) := \int
        |\varphi'(x)|^2\,\rd x + \langle q,|\varphi|^2\rangle,
\]
where $\langle\cdot,\cdot\rangle$ is the pairing between
$H^{-1}_\mathrm{loc}(\mathbb{R})$ and $H^1_\mathrm{comp}(\mathbb{R})$.
In \cite{KapPerShuTop05} it was shown that
if $q$ is any real-valued distribution in $H^{-1}_\mathrm{loc}(\mathbb{R})$ for which the Schr\"{o}dinger form~$\mathfrak{q}$ is nonnegative, then $q$ may be presented as $q=Mu$ for a function $u\in L^2_\mathrm{loc}(\mathbb{R})$
that need not be unique.
We will call such a potential $q$ a Miura potential, and we will call any function $u$ with $q=Mu$ a {\it Riccati representative} for $q$.
It is not difficult to see that two Riccati representatives for a given $q$ differ by a continuous function.
Moreover, any Riccati representative $u$ is the logarithmic derivative of a positive distributional solution to the zero-energy Schr\"{o}dinger equation for $q$, and the differential expression $\ell(y):=-y''+q$ can be written as 
\begin{equation}\label{eq:diff-eq-fact}
 \ell(y) = -y'' + qy = -(y^{[1]})' - u y^{[1]} = -\Bigl(\frac{d}{d x}+u\Bigr)\Bigl(\frac{d}{d x}-u\Bigr),
\end{equation}
with $y^{[1]} = y' - uy$ being the quasi-derivative of~$y$.

In this paper, we will study inverse scattering for the subset of the Miura potentials which admit a real-valued Riccati representative $u\in X(=L^1(\mathbb{R}_+)\cap L^2(\mathbb{R}_+))$.
In this case, such a Riccati representative is unique, so we may parameterize the potentials by their Riccati representatives from~$X$. 
%We will consider equations~\eqref{eq:EDSE} with the energy-dependent boundary condition
%\begin{equation}\label{eq:DBC}
%    \sin\alpha\, y^{[1]}(0,k) + k \cos\alpha\, y(0,k)=0,
%\end{equation}
%where the {\it quasi-derivative} $y^{[1]}$ is defined as $y^{[1]}(x,k):=y'(x,k)-u(x)y(x,k)$.

%%%%%%%%%%%%%%%%%%%%%%%%%%%%%%%%%%%%%%%%%%%%%%%%%

\subsection{Spectral properties of the problem~\eqref{eq:EDSE}--\eqref{eq:bc}}\label{ssec:spectral}
Since equation~\eqref{eq:EDSE} depends both on the spectral parameter~$k$ and its square~$k^2$, it can naturally be placed within the theory of quadratic operator pencils~\cite{Mar:1988}. Recall that under our assumptions, $q$ in~\eqref{eq:EDSE} is a Miura potential, i.e., that $q = u'+u^2$ for some Miura representative~$u \in X$. Denote by $I$ the identity operator in $L^2(\bR_+)$, by $A_z$, $z\in\bC$, the differential operator in $L^2(\bR_+)$ given by~\eqref{eq:diff-eq-fact} on the domain
\begin{align*}
	 \dom(A_z) = \Bigl\{ y \in AC_{\mathrm{loc}}(\bR_+) \cap L^2(\bR_+) \mid & y^{[1]} \in AC_{\mathrm{loc}}(\bR_+), \ell(y) \in L^2(\bR_+), \\
	 		  & \sin\alpha\, y^{[1]}(0) + z \cos\alpha\, y(0)=0\Bigr\}
\end{align*}
and by $B$ the operator of multiplication by the function~$p$.
For every  $z\in\bC$, set 
\[
	T(z):= A_z + zB -z^2I;
\]
$T$ is a quadratic operator pencil related to the problem~\eqref{eq:EDSE}--\eqref{eq:bc}. It is an operator-valued function of $z\in\bC$; for every $z$, $T(z)$ is a well-defined operator on the domain $\dom(T(z)) = \dom(A_z)$.

We call a number $z_0\in \bC$ \emph{resolvent point} of the operator pencil~$T$ if the operator $T(z_0)$ is boundedly invertible in~$L^2(\bR_+)$, i.e., if $0$ is in the resolvent set of the operator~$T(z_0)$. Next, $z_0\in \bC$ is an \emph{eigenvalue} of $T$ if $0$ is an eigenvalue of the operator~$T(z_0)$; in that case, every non-trivial function~$y\in\dom(T)$ satisfying $T(z_0)y=0$ is called a corresponding \emph{eigenfunction}. Similarly, $z_0\in \bC$ is a point of the \emph{essential spectrum} of $T$ if $0$ is in the essential spectrum of the operator~$T(z_0)$. 

Under the assumptions on the functions $u$ and $p$, for every $z\in\bC$ the operator $T(z)$ is a relatively compact perturbation of the operator $A_z - z^2I$~\cite{Wei87}; as a result, we conclude that the essential spectrum of the operator pencil~$T$ coincides with the whole real line. Next, eigenvalues come in complex conjugate pairs: indeed, if $z$ is an eigenvalue of~$T$ with corresponding eigenfunction~$y$, then $\overline{z}$ is an eigenvalue of~$T$ with an eigenfunction~$\overline{y}$. 

We observe that if $q$ is a Miura potential, then $T$ has no non-real eigenvalues. Indeed, integration by parts on account of~\eqref{eq:diff-eq-fact} shows that, for every $z\in\bC$ and every $y \in \dom(A_z)$, we have 
\[
	(A_z y,y) = \|y^{[1]}\|^2 + y^{[1]}(0)\overline{y}(0) = \|y^{[1]}\|^2 - z \cot\alpha |y(0)|^2 
\]
if $\sin\alpha\ne0$ or $(A_z y,y) = \|y^{[1]}\|^2$ otherwise. Assuming $z = \sigma + i\tau$ with $\tau\ne0$ is an eigenvalue of $T$ with an eigenfunction~$y$ and separating the real and imaginary parts in the scalar product $(T(z)y,y)$, we conclude that
\[
	(By,y) -\cot\alpha|y(0)|^2 = 2\sigma \|y\|^2 
\]
and
\[	
	\|y^{[1]}\|^2  + \sigma (By,y) - \sigma\cot\alpha|y(0)|^2  + (\tau^2-\sigma^2) \|y\|^2 =0
\]
in the case $\sin\alpha \ne 0$; otherwise, the terms with $\cot\alpha$ should be omitted. Both cases lead to the contradiction $\|y^{[1]}\|^2 + |z|^2(y,y)=0$, thus justifying absence of non-real eigenvalues.

Conversely, if the operator pencil~$T$ has no non-real eigenvalues, then under some mild conditions~\cite{Mar:1988} we have $A_0>0$ and then $q$ must be a Miura potential. Therefore, under the assumption that the energy-dependent Schr\"{o}dinger equation~\eqref{eq:EDSE} has no eigenvalues (there can be no real eigenvalues due to the asymptotics of the Jost solutions, see Subsection~\ref{ssec:Jost}), it is natural to assume that the potential~$q$ is a Miura potential $q = Mu$ for some function~$u$.

%%%%%%%%%%%%%%%%%%%%%%%%%%%%%%%%%%%%%%%%%%%

\subsection{Transformation to a ZS-AKNS system}

Next we establish a relation between solutions to the energy-dependent Schr\"odinger equation~\eqref{eq:EDSE} with Miura potentials~$q=Mu$ and those for some Dirac-type equation. Namely, take any solution $y(\cdot,k)$ of~\eqref{eq:EDSE} for a nonzero~$k$ and set $y_2(\cdot,k):=y(\cdot,k)$ and
$y_1(\cdot,k):=y^{[1]}(\cdot,k)/k$, i.e.,
\begin{equation}\label{eq:y1}
    y_1(\cdot,k) = \frac{y_2'(\cdot,k) - u y_2(\cdot,k)}k;
\end{equation}
here $u$ is the Riccati representative of the potential~$q$ belonging to~$X$. Then $y_1$ and $y_2$ satisfy the system
\begin{alignat*}{5}
    y_2'\,&& & -& u y_2  &=\,&   k y_1,\\
    -y_1'\, &-& u y_1 &+& 2p y_2  &=\,&  k y_2,
\end{alignat*}
i.e., the column vector $\by(\cdot,k):=\bigl(y_1(\cdot,k),y_2(\cdot,k)\bigr)^\mathrm{t}$ satisfies the non-canonical Dirac system
\begin{equation}\label{eq:sys.y}
    \sigma_2\mathbf{y}'+P(x)\mathbf{y}=k\mathbf{y},
\end{equation}
where
\begin{equation}\label{eq:potential-P}
    \sigma_2:= \begin{pmatrix} 0&1\\ -1&0 \end{pmatrix}
        \qquad \text{and} \qquad
    P:= \begin{pmatrix} 0&-u\\ -u&2p \end{pmatrix}.
\end{equation}
The boundary condition~\eqref{eq:bc} for $\by$ read
\begin{equation}\label{eq:bcy}
    \sin \alpha\, y_1(0,k) + \cos \alpha\,  y_2(0,k) = 0.
\end{equation}

Conversely, if $\by(\cdot,k)=\bigl(y_1(\cdot,k),y_2(\cdot,k)\bigr)^\mathrm{t}$ solves the above Dirac system, then the first component $y_1$ verifies~\eqref{eq:y1}, while the second component $y:=y_2$ solves the equation
\[
    -\Bigl(\frac{d}{d x}+u\Bigr)\Bigl(\frac{d}{d x}-u\Bigr)y
            +2k p(x)y=k^2y,
\]
which coincides with the energy-dependent Schr\"{o}din\-ger equation~\eqref{eq:EDSE} with $q = Mu = u'+u^2$.

Hence in order to construct the Jost solutions for the energy-dependent
Schr\"{o}din\-ger equation~\eqref{eq:EDSE} with a Miura potential $q = B(u)$, it suffices to study the Jost solutions for the corresponding Dirac system~\eqref{eq:sys.y}. To this end we make a substitution
\[
    \by = U \bz = \begin{pmatrix} i & -i \\ 1 & 1\end{pmatrix}\bz
\]
with a vector~$\bz:=(z_1,z_2)^{\mathrm{t}}$
and after straightforward calculations recast~\eqref{eq:sys.y} as
\begin{equation}\label{eq:sys.z}
    \mathbf{z}'+Q(x)\mathbf{z}=ik\sigma_3\mathbf{z},
\end{equation}
with
\begin{equation}\label{eq:potential-Q}
    \sigma_3:= \begin{pmatrix} 1&0\\ 0&-1 \end{pmatrix},
        \qquad
    Q:= \begin{pmatrix} ip   &-u+ip\\ -u-ip&-ip \end{pmatrix};
\end{equation}
the boundary condition~\eqref{eq:bcy} then becomes
\begin{equation*}
%\label{eq:bcz}
    e^{i\alpha} z_1(0,k) + e^{-i\alpha} z_2(0,k) = 0.
\end{equation*}

We note that the system~\eqref{eq:sys.z} is still not in a canonical form, but it takes the ZS-AKNS canonical form after the substitution $\bz(x):= \exp(i \varphi(x) \sigma_3)\bw(x)$ with $\varphi(x):=\int_x^\infty p(t)\,\rd t$: indeed, the equation for
 $\bw(\cdot,k) = \bigl(w_1(\cdot,k),w_2(\cdot,k)\bigr)^\textrm{t}$
reads
\begin{equation}\label{eq:sys.w}
    \bw' + V(x) \bw = ik \sigma_3 \bw
\end{equation}
with $V(x) := \bigl(Q(x) - ip(x)\sigma_3\bigr)e^{2i \varphi(x)\sigma_3}$, i.e., with
\begin{equation}\label{eq:potential-V}
    V(x) =  \begin{pmatrix} 0 & v(x) \\
                \overline{v(x)} & 0 \end{pmatrix},
\end{equation}
and $v(x):=(-u(x)+ip(x))e^{-2i\varphi(x)}$. We observe that under the assumption made on~$p$ and $u$ the function~$v$ belongs to the space $X$.
Set~$p_0:=\varphi(0) = \int_0^\infty p(t)\,dt$ and $\beta:=\alpha+ p_0$; then $\bw$ should satisfy the boundary condition
\begin{equation}\label{eq:bcw}
    e^{i\beta}w_1(0,k) + e^{-i\beta}w_2(0,k) = 0.
\end{equation}

%%%%%%%%%%%%%%%%%%%%%%%%%%%%%%%%%%%%%%%%%%%%%%%%

\subsection{Jost solutions}\label{ssec:Jost} Properties of the Jost solutions for the canonical ZS-AKNS system~\eqref{eq:sys.w} with $v \in X$ were thoroughly studied in the papers~\cite{FHMP09,HryMan13}. Namely, it was proved therein that for every real~$k\in\mathbb{R}$ the equation
\begin{equation}\label{eq:ZS-Psi}
    \Psi' + V \Psi=ik\sigma_3 \Psi
\end{equation}
with $V$ of~\eqref{eq:potential-V} and a complex-valued $v \in X$
has a unique matrix $2\times 2$ solution $\Psi$ obeying the asymptotics~$\Psi(x,k) = e^{ik x \sigma_3}(1 + o(1))$. Next, there exists a matrix $2\times 2$ kernel $\Gamma$ such that
\begin{equation}\label{eq:JS-Psi}
    \Psi(x,k)  = e^{ikx\sigma_3} + \int_0^\infty \Gamma(x,\zeta) e^{2ik\zeta}\,d\zeta;
\end{equation}
moreover, $\Gamma$ has the property that for every fixed~$x\ge0$ the function~$\Gamma(x,\cdot)$ belongs to $X \otimes M_2(\bC)$ and depends continuously therein on~$x\ge0$.

Therefore, one concludes that the corresponding non-canonical ZS-AKNS system
\begin{equation*}%\label{eq:ZS-Phi}
    \Phi' + Q \Phi=ik\sigma_3 \Phi
\end{equation*}
with $Q$ of~\eqref{eq:potential-Q} has a unique matrix $2\times 2$ solution $\Phi$ obeying the asymptotics~$\Phi(x,k) = \re^{\ri k x \sigma_3}(1 + o(1))$; clearly, this solution satisfies $\Phi(x,k) = \exp(i\varphi(x)\sigma_3)\Psi(x,k)$.
Further,
$Y(x,k):= U \Phi(x,k)$ is a matrix-valued solution of the Dirac equation~\eqref{eq:sys.y}; its first column $\by$ obeys the asymptotics
\[
    \by(x,k) = U \exp( i\varphi(x)\sigma_3) \bpsi_1(x,k)
             = \begin{pmatrix} i  \\
                    1 \end{pmatrix} e^{i k x}(1+ o(1))
\]
as $x\to\infty$, where $\bpsi_1$ is the first column of the matrix solution~$\Psi$. The second component of~$\by=(y_1,y_2)^{\mathrm{t}}$ is the Jost solution of the energy-dependent Schr\"odinger equation~\eqref{eq:EDSE} with $q= Mu$, while $y_1 = y_2^{[1]}/k$. The integral representation of $\Psi$ leads to the following representation of $\by$:
\[
    \by(x,k) = e^{ik (x + \varphi(x))} \binom{i}{1}
        + \int_0^\infty \mathbf{g}(x,\zeta)e^{2i k \zeta}\,d\zeta,
\]
where $\mathbf{g}=(g_1,g_2)^{\mathrm{t}}$ is a function with the property that $g_j(x,\cdot)$ belongs to $X$ and depends therein continuously on~$x\ge0$.

\begin{lemma}
For every nonzero $k\in\bR$, the Jost solution $f(\cdot,k)$ of the energy-dependent Schr\"odinger equation~\eqref{eq:EDSE} with a Miura potential $q=Mu$, $u\in X$, exists and is unique.
\end{lemma}

\begin{proof}
The above considerations produce a particular Jost solution $f(\cdot,k)$ of~\eqref{eq:EDSE} in a constructive manner from the Jost solution of the companying ZS-AKNS system~\eqref{eq:ZS-Psi}.
To justify uniqueness of the Jost solution, it is enough to observe that $\overline{f(\cdot,k)}$ is another solution of~\eqref{eq:EDSE} that is linearly independent of~$f(\cdot,k)$ as it obeys a different asymptotics at~$+\infty$. Clearly, no linear combination $\overline{f(\cdot,k)} + c f(\cdot,k)$ is asymptotic to $\re^{\ri kx}$ at~$+\infty$, whence the Jost solution is unique.
\end{proof}

%%%%%%%%%%%%%%%%%%%%%%%%%%%%%%%%%%%%%%%%%%%%%%%%%%%%%%%%%%%%%%%%%

\section{Direct scattering problem}\label{sec:direct}

%%%%%%%%%%%%%%%%%%%%%%%%%%%%%%%%%%%%%%%%%%%%%%%%%%%%%%%%%%

\subsection{The scattering function}

Recall that the scattering function~$S$ for the problem~\eqref{eq:EDSE}, \eqref{eq:bc} is equal to~$s(k) /\overline{s(k)}$, where the Jost function~$s$ is given by~\eqref{eq:Jost-funct}. In terms of the Jost solution $\by$ of the Dirac system~\eqref{eq:sys.y}, the Jost function reads
\[
    s(k) = k\bigl(\sin\alpha\, y_1(0,k) + \cos\alpha\, y_2(0,k)\bigr).
\]
The scattering function~$S$ then takes the form
\[
    S(k) = - \frac{\sin\alpha\,y_1(0,k) + \cos\alpha\,y_2(0,k)}%
            {\sin\alpha\,\overline{y_1(0,k)} + \cos\alpha\, \overline{y_2(0,k)}};
\]
it is uniquely defined by the requirement that the function~$\by(\cdot,k) + S(k)\overline{\by(\cdot,k)}$ is a solution of the system~\eqref{eq:sys.y} satisfying the boundary condition~\eqref{eq:bcy}.

Making the substitution $\by=U\exp(i\varphi(x)\sigma_3)\bw$, one gets the following representation of~$s$,
\[
    s(k) = k\bigl(\re^{\ri(\alpha+p_0)}w_1(0,k) + \re^{-\ri(\alpha+p_0)}w_2(0,k)\bigr),
\]
in terms of the Jost solution~$\bw=(w_1,w_2)^{\mathrm{t}}$ of the ZS-AKNS system~\eqref{eq:sys.w}; respectively, the scattering function~$S$ becomes
\begin{equation}\label{eq:S-via-w}
    S(k)
    = - \frac{\re^{\ri(\alpha+p_0)}w_1(0,k) + \re^{-\ri(\alpha+p_0)}w_2(0,k)}
           {\re^{-\ri(\alpha+p_0)}\overline{w_1(0,k)} + \re^{\ri(\alpha+p_0)}
           \overline{w_2(0,k)}}.
\end{equation}

%%%%%%%%%%%%%%%%%%%%%%%%%%%%%%%%%%%%%%%%%%%%%%%%%%%%%%%%%

\subsection{Scattering for ZS-AKNS systems}

In view of the results proved in~\cite{HryMan13}, the above formula for the scattering function~$S$ suggests that it is also a scattering function for some canonical ZS-AKNS system.

Indeed, consider the $2\times 2$ matrix Jost solution~$\Psi=(\psi_{jl})_{j,l=1}^2$ of the ZS-AKNS system~\eqref{eq:ZS-Psi}. The column vectors~$\bpsi_1$ and $\bpsi_2$ are linearly independent and solve equation~\eqref{eq:sys.w}; therefore, there is a linear combination of the form $\bpsi_1(\cdot,k) + c \bpsi_2(\cdot,k)$ satisfying the boundary condition~\eqref{eq:bcw}. The coefficient $c$ is equal to
\begin{equation}\label{eq:scat-psi}
    -\frac{\re^{\ri\beta}\psi_{11}(0,k)+ \re^{-\ri\beta}\psi_{21}(0,k)}
    {\re^{-\ri\beta}\overline{\psi_{11}(0,k)}+ \re^{\ri\beta}\overline{\psi_{21}(0,k)}}
\end{equation}
and is called the scattering function for the problem~\eqref{eq:sys.w}, \eqref{eq:bcw}.

It is proved in~\cite{HryMan13} that the scattering functions for canonical ZS-AKNS systems~\eqref{eq:sys.w} with potentials~\eqref{eq:potential-V} for $v\in X$ and subject to the boundary condition~\eqref{eq:bcw} belong to the following set~$\mathcal{S}$.

\begin{definition}
We say that a function $S\colon \mathbb{R}\to\mathbb{C}$ belongs to the class $\mathcal{S}$ if and only if
\begin{itemize}
\item[(1)] there are $F\in L^1(\mathbb{R})\cap L^2(\mathbb{R})$ and $\gamma\in[0,\pi)$ such that for all $k\in\mathbb{R}$ it holds
\begin{equation}\label{intrep:S}
    S(k)=\re^{2\ri\gamma}+\int_{-\infty}^\infty F(\zeta)\re^{2\ri k\zeta}\,\rd\zeta;
\end{equation}
\item[(2)] the function $S$ is unimodular, i.e. $(S(k))^{-1}=\overline{S(k)}$ for all real $k$;
\item[(3)] the winding number $W(S)$ of the function~$S$ is equal to zero.
\end{itemize}
\end{definition}

We recall that the \emph{winding number} $W(f)$ of a function $f$ that is continuous on the extended real line $\overline{\bR}$ (i.e., $f$ is continuous on~$\bR$ and possesses finite and equal limits $\lim_{k\to\pm\infty}f(k)$) is the integer number equal to
\[
    W(f):=\frac{\log f(+\infty) - \log f(-\infty)}{2\pi i}.
\]
One of the main results of~\cite{HryMan13} can be formulated as follows.

\begin{proposition}\label{pro:scat-ZS-AKNS}
The function $S$ is the scattering function of the problem
\eqref{eq:sys.w}, \eqref{eq:bcw} corresponding to some $(v,\beta)\in X\times[0,\pi)$ if and only if $S$ belongs to $\mathcal{S}$ and
the number $\gamma$ in its integral representation \eqref{intrep:S} is equal to $\beta$.
\end{proposition}

Moreover, the paper~\cite{HryMan13} gives an algorithm based on the Marchenko equation that for a given function $S \in \mathcal{S}$ constructs a ZS-AKNS system (i.e., the function~$v$ in the potential~$Q$ and the number~$\beta$ in the boundary condition) whose scattering function is this~$S$.

%%%%%%%%%%%%%%%%%%%%%%%%%%%%%%%%

\section{Inverse scattering for energy-dependent Schr\"odinger equations}\label{sec:inverse}

The representation~\eqref{eq:S-via-w} of the scattering function for the energy-dependent Schr\"odinger equation~\eqref{eq:EDSE} in terms of the Jost solution~$\bw$ of the companying canonical ZS-AKNS system~\eqref{eq:sys.w} reveals a close connection to the scattering functions of the latter.

We para\-me\-t\-rise the set of problems \eqref{eq:EDSE}, \eqref{eq:bc} by the triples $(u,p,\alpha)\in X\times X\times[0,\pi)$ of Riccati representatives $u$ of the potential~$q$, the potential~$p$, and the constant $\alpha$ in the boundary conditions.
Our main result gives a complete description of the scattering functions for equations~\eqref{eq:EDSE} subject to the boundary conditions~\eqref{eq:bc}:

\begin{theorem}\label{thm:main}
A function $S$ is the scattering function of the problem
\eqref{eq:EDSE}, \eqref{eq:bc} corresponding to some $(u,p,\alpha)\in X\times X\times[0,\pi)$ if and only if $S$ belongs to $\mathcal{S}$ and the number $\gamma$ in its integral representation \eqref{intrep:S} is equal to $\gamma_0:=(\alpha+p_0) (\hskip-6pt\mod\pi)$.
\end{theorem}

The proof of this theorem can naturally be split into three parts: necessity (i.e., characterization of the scattering functions), uniqueness (different energy-dependent Schr\"odinger equations lead to different scattering functions), and existence (every function in~$\mathcal{S}$ is the scattering function for some problem under consideration).

Also, as shown in Section~\ref{sec:prelim}, the energy-dependent Schr\"odinger equations~\eqref{eq:EDSE} are in one-to-one correspondence with the related Dirac systems~\eqref{eq:sys.y} in non-canonical form. Therefore one can study non-canonical Dirac systems~\eqref{eq:sys.y} instead of~\eqref{eq:EDSE}; the corresponding boundary conditions are given by~\eqref{eq:bcy}.

\begin{proof}[Proof of necessity]
On account of the representation~\eqref{eq:S-via-w} and Proposition~\ref{pro:scat-ZS-AKNS} we conclude that the scattering function~$S$ of every non-canonical Dirac system~\eqref{eq:sys.y}, \eqref{eq:bcy}  corresponding to some $(u,p,\alpha)\in X\times X\times[0,\pi)$ belongs to~$\mathcal{S}$.
\end{proof}

\begin{proof}[Proof of uniqueness]
Assume that there exist two Dirac-type problems~\eqref{eq:sys.y} with potentials $P_j$, $j=1,2$, of the form~\eqref{eq:potential-P}, i.e., with
\[
    P_j:= \begin{pmatrix} 0&-u_j\\ -u_j&2p_j \end{pmatrix},
        \qquad u_j,p_j \in X,
        \qquad j = 1,2,
\]
and subject to the boundary condition~\eqref{eq:bcy} with the respective $\alpha_j\in[0,\pi)$ instead of~$\alpha$, sharing the same scattering function~$S$. We shall prove that $P_1 = P_2$ and $\alpha_1 = \alpha_2$.

Indeed, as explained in the previous section, the scattering problem for every Dirac-type system~\eqref{eq:sys.y} can be transformed to the scattering problem for a canonical ZS-AKNS system~\eqref{eq:sys.w}; moreover, the scattering functions for both systems will be the same.

Set~$\varphi_j(x):=\int_x^\infty p_j(t)\,dt$ and $v_j:= (-u_j + i p_j)e^{-2i\varphi_j}$, $j=1,2$.
Then the above transformation gives two canonical ZS-AKNS systems~\eqref{eq:sys.w} with potentials~$V_j$, $j=1,2$, of the form~\eqref{eq:potential-V}, in which $v$ is replaced by~$v_j$; the corresponding boundary conditions have the form~\eqref{eq:bcw} with some (plausibly different) $\beta_1$ and $\beta_2$ from the interval~$[0,\pi)$.
Since the scattering functions for these ZS-AKNS systems are the same, Proposition~\ref{pro:scat-ZS-AKNS} yields the equality $v_1 = v_2$ and $\beta_1 = \beta_2$. In particular,
\begin{equation}\label{eq:r}
    |-u_1(x) + i p_1(x)| = |-u_2(x) + i p_2(x)| =:r(x),
\end{equation}
and there exist functions~$\eta_j(x)$ such that
\[
    -u_j(x) + i p_j(x) = r(x) e^{i\eta_j(x)}, \qquad j=1,2.
\]

Observe that the equality $v_1 = v_2$ yields the relation
$\eta_1 - 2\varphi_1 = \eta_2 - 2\varphi_2 (\hskip-6pt\mod 2\pi)$, so that
$\eta_2 - \eta_1 = 2(\varphi_2 - \varphi_1)(\hskip-6pt\mod 2\pi)$. Since $p_j = r \sin \eta_j$, we get
\[
    p_2 - p_1 = r (\sin\eta_2 - \sin\eta_1)
              = 2r \sin (\varphi_2 - \varphi_1) \cos (\varphi_2 - \varphi_1 + \eta_1).
\]
Next we set $\theta:=\varphi_2 - \varphi_1$ and recall that $p_j = - \varphi'_j$ to get the following equation for~$\theta$ on $\bR_+$:
\begin{equation}\label{eq:theta}
    \theta'=- 2r\sin \theta \cos(\theta + \eta_1).
\end{equation}
By the definition of $\varphi_j$, the function $\theta$ vanishes at~$+\infty$, i.e.,
\begin{equation}\label{bc:theta}
\theta(x)\to0,\qquad x\to\infty.
\end{equation}

Now we claim that the problem \eqref{eq:theta}--\eqref{bc:theta} possesses only trivial solution $\theta\equiv0$. Observe that $\theta$ solves the integral equation
\[
    \theta(x)= 2\int_x^\infty r(t)\sin\theta(t)\cos(\theta(t)+ \eta_1(t))\,d t;
\]
thus, we get the bound
\begin{equation}\label{ineq:theta}
    |\theta(x)|\le 2\int^\infty_x |r(t)||\theta(t)|\,dt
\end{equation}
holding for all positive $x$.
Introduce the least monotonically decreasing majorant of the function $\theta$:
\[
    \Theta(x):=\max\limits_{t\geq x}|\theta(t)|.
\]
By virtue of \eqref{ineq:theta} we derive the estimate
\[
    \Theta(x) \le 2\int_x^\infty |r(t)||\theta(t)|\,\rd t
              \le 2\Theta(x)\int^\infty_x |r(t)|\,\rd t.
\]
Due to the definition of $r$ (see~\eqref{eq:r}) and the inclusions $u_j, p_j \in X$, $r$ is integrable, and thus there is an $x_0\ge0$ such that
\[
    \int_{x_0}^\infty |r(t)|\,\rd t < \frac14;
\]
then we find that $\Theta(x_0)=0$ and so $\theta$ is zero for $x\ge x_0$.

Next we show that $\theta$ is zero on $\bR_+$.
Assume the contrary and denote by $x^*$ the greatest lower bound of all $x$ such that $\theta=0$ a.e.\ on $(x,\infty)$. Then $x^*>0$ by assumption, and
since the function $r$ is integrable, there exists $x_*\in(0,x^*)$ such that
\[
    \int_{x_*}^{x^*} |r(t)|\,\rd t\leq\frac14.
\]
Thus
\[
    \Theta(x_*)
        \leq2\Theta(x_*)\int_{x_*}^{x^*} |r(t)|\,\rd t
        \leq\frac12\Theta(x_*),
\]
which implies that $\Theta$ vanishes a.e.\ on $(x_*,x^*)$. This contradicts the definition of~$x^*$ and the assumption that~$x^*>0$; therefore $x^*=0$ and $\Theta$ vanishes on $\mathbb{R}_+$.

As a result, we get that $\theta$ is zero on $\mathbb{R}_+$, i.e., that $\varphi_1= \varphi_2$ on $\bR_+$, whence $u_1=u_2$ and $p_1=p_2$ on~$\bR_+$ by virtue of the relation $v_1=v_2$.
Now the formula~\eqref{eq:S-via-w} shows that $\alpha_1 = \alpha_2 (\hskip-6pt\mod\pi)$, and
the proof is complete.
\end{proof}

\begin{proof}[Proof of existence]
Given an arbitrary~$S$ in $\mathcal{S}$, we use Proposition~1 and find a unique canonical ZS-AKNS system~\eqref{eq:sys.w} and the number~$\beta$ in the boundary conditions~\eqref{eq:bcw} having this~$S$ as its scattering function. The potential~$V$ of this ZS-AKNS system has the form~\eqref{eq:potential-V} with some complex-valued function~$v \in X$.

Our task is to represent this function~$v$ as $(-u+i p)e^{-2i\varphi}$ for some functions~$u$ and $p$ from~$X$ and with $\varphi(x):=\int_x^\infty p(t)\,dt$.
In other words,
\[
    -u(x) + i p(x) = v(x)e^{2i\varphi(x)},
\]
and, equating the imaginary parts and recalling that $p = -\varphi'$, we get the differential equation
\begin{equation}\label{eq:de-gamma}
    \varphi'=-(\myRe  v)\sin2\varphi - (\myIm v)\cos2\varphi
\end{equation}
for the function~$\varphi$. By the definition of~$\varphi$, it should vanish at $+\infty$.

First we prove that the above equation has a unique solution vanishing at~$+\infty$.

\begin{lemma}
Suppose that $v_1$ and $v_2$ are real-valued functions in $L^1(\mathbb{R}_+)$.
Then the equation
\begin{equation}\label{eq:de-phi}
    \phi'= - v_1\sin2\phi - v_2\cos2\phi
\end{equation}
admits a unique real-valued solution $\phi$ on~$\bR_+$ tending to zero as $x$ tends to infinity.
\end{lemma}

\begin{proof}
Fix a number $x_0\ge0$ such that
\[
    \int_{x_0}^\infty \bigl(|v_1(t)| + |v_2(t)|\bigr)\,dt < \frac14
\]
and consider the function space
\[
    Y=\Big\{f\in AC(x_0,\infty)\mid \forall x\ge x_0,\
            \overline{f(x)} = f(x),\ \lim_{x\to\infty}f(x) =0 \Big\}
\]
with the norm
\[
    \|f\|_Y :=\int_{x_0}^\infty|f'(t)|\,dt.
\]
We claim that the space $Y$ is a Banach space over the field~$\bR$. Indeed, if $(f_n)_{n\in\bN}$ is a Cauchy sequence in~$Y$, then the sequence of the derivatives $(f_n')$ is Cauchy in~$L^1(x_0,\infty)$ and thus converges therein to some real-valued function~$f_0 \in L^1(x_0,\infty)$; now
\[
    f(x):= - \int_x^\infty f_0(t)\,dt
\]
is the limit of~$(f_n)$ in~$Y$.

In the Banach space $Y$, we introduce the nonlinear integral operator
\[
    \big(Tf\big)(x):=
            \int_x^\infty \big(u_1(t)\sin2f(t)+u_2(t)\cos2f(t)\big)\,dt.
\]
Observe that if $\phi$ solves~\eqref{eq:de-phi}, then $\phi'$ is integrable over~$\bR_+$, and upon integrating both sides of~\eqref{eq:de-phi} we see that $\phi$ is a fixed point of the operator $T$. Conversely, a fixed point of~$T$ is also a solution of~\eqref{eq:de-phi} for $x\ge x_0$ vanishing at $+\infty$. We therefore look for fixed points of the mapping~$T$.

Clearly, $T$ maps~$Y$ into~$Y$. In view of the inequalities
\begin{align*}
    |\sin2f_1(t)-\sin2f_2(t)|, |\cos2f_1(t)-\cos2f_2(t)|
            \leq 2|f_1(t) -f_2(t)| \leq 2\|f_1-f_2\|_Y,
\end{align*}
we conclude that
\begin{equation*}
    \|Tf_1-Tf_2\|_Y
        \leq 2\|f_1-f_2\|_Y
            \int_{x_0}^\infty\big(|u_1(t)|+|u_2(t)|\big)\,\rd t
        \leq \frac12\|f_1-f_2\|_Y
\end{equation*}
so that $T$ is a contraction in~$Y$. Next we find that
\[
    \|T0\|_Y= \int_{x_0}^\infty |u_2(t)|\,\rd t \le \frac14,
\]
and if $f$ is an arbitrary element of the closed unit ball in $Y$, then
\[
    \|Tf\|_Y \leq \|T0\|_Y + \|Tf-T0\|_Y
             \leq \frac14+\frac12\|f\|_Y
             \leq\frac34.
\]

Thus $T$ is a contraction in~$Y$ that maps its closed unit ball into itself and therefore by the Banach fixed point theorem there is a unique $\phi$ in the unit ball of~$Y$ such that $\phi = T\phi$. This fixed point $\phi$ solves equation~\eqref{eq:de-phi} on $(x_0,\infty)$. By the standard existence and uniqueness theorems, this $\phi$ can be extended to a solution of~\eqref{eq:de-phi} over the whole half-line~$\bR_+$. The proof is complete.
\end{proof}

In view of the above lemma, equation~\eqref{eq:de-gamma} has a solution~$\varphi$ vanishing at infinity. Set
\begin{equation}\label{eq:u-p-solution}
\begin{aligned}
    u&:= - (\myRe v) \cos2\varphi + (\myIm v)\sin2\varphi,\\
    p&:=     (\myRe  v)\sin2\varphi + (\myIm v)\cos2\varphi;
\end{aligned}
\end{equation}
then $u$ and $p$ are real-valued, belong to~$X$ and, moreover, satisfy the relation
$v = (-u+ip)e^{-2i\varphi}$. Finally, we define the number $\alpha\in[0,\pi)$ from the equality $\alpha + \varphi(0)=\beta(\hskip-6pt\mod\pi)$.
As shown in Section~\ref{sec:direct}, the scattering function of~\eqref{eq:scat-psi} for the ZS-AKNS system constructed at the beginning of this proof is also a scattering function for the non-canonical Dirac system~\eqref{eq:sys.y} with potential~$P$ of~\eqref{eq:potential-P} with the $u$ and $p$ of~\eqref{eq:u-p-solution} and with the number $\alpha$ in the boundary condition defined as $\beta-\varphi(0)$ modulo~$\pi$. Therefore the scattering function of the so constructed non-canonical Dirac system is the function $S\in\mathcal{S}$ we have started with. The proof is complete.
\end{proof}

Analyzing the proof of the above theorem, we can suggest an explicit algorithm reconstructing the triple $(u,p,\alpha)$ from $S$.
Namely, given an arbitrary element $S$ of $\mathcal{S}$, we construct the energy-dependent Schr\"{o}dinger equation \eqref{eq:EDSE} with a Riccati representative $u\in X$ of the potential~$q$, the potential~$p\in X$, and the constant $\alpha\in[0,\pi)$ in the boundary conditions \eqref{eq:bc} in the following way:
\begin{itemize}
    \item[(1)] construct the ZS-AKNS system \eqref{eq:sys.w}, \eqref{eq:bcw} with a potential $v\in X$ and a constant $\beta\in[0,\pi)$ whose scattering function coincides with $S$ (see~\cite{HryMan13}):

  \item [] recover $\beta$ and $F$ from the integral representation \eqref{intrep:S},
form the matrix-valued function
\begin{equation*}
    \Omega(x):= \begin{pmatrix} 0&\overline{F(x)}\\ F(x)&0 \end{pmatrix}
\end{equation*}
and consider the corresponding {\it Marchenko equation}
\begin{equation*}
    \Gamma(x,\zeta) +\Omega(x+\zeta)
        + \int_0^\infty \Gamma(x,t)\Omega(x+t+\zeta)\,dt=0
\quad\text{for a.e.}\quad \zeta>0
\end{equation*}
for the kernel~$\Gamma=(\Gamma_{jk})_{j,k=1}^2$ of the Jost solution \eqref{eq:JS-Psi}; this Marchenko equation is uniquely soluble and gives the potential $v$ by \[v(x)=-\Gamma_{12}(x,0)=-\overline{\Gamma_{21}(x,0)};\]

\item [(2)] having $v$, find a unique solution $\varphi$ of the equation \eqref{eq:de-gamma} vanishing at infinity;

\item [(3)] define the potentials $u\in X$ and $p\in X$ via formulas \eqref{eq:u-p-solution}, and the constant $\alpha$ in the boundary condition as $\beta-\varphi(0)$ modulo~$\pi$.

\end{itemize}

%%%%%%%%%%%%%%%%%%%%%%%%%%%%%%%%%%%%%%%%%%%%%%%%%%

\noindent
{\it Acknowledgements}
The authors thank anonymous referees for useful remarks that have led to the improvement of the manuscript.
The research of~R.~H. was partially supported by the Centre for Innovation
and Transfer of Natural Sciences and Engineering Knowledge at the University
of Rzesz\'ow.
The research of~S.~M. was supported by  the European Union within the project ``Support for research teams on CTU'' CZ.1.07/2.3.00/30.0034.

%%%%%%%%%%%%%%%%%%%%%%%%%%%%%%%%%%%%%%%%%%%%%%%%%%%%%

%%%%%%%%%%%%%%%%%%%%%%%%%%%%%%%%%%%%%%%%%%%%%%%%%%%%%

\end{document}